\documentclass[11pt, a4paper]{scrartcl}
\usepackage{ngerman}
\usepackage{amsmath}
\usepackage{amsthm}
\usepackage{amsfonts}
\usepackage{amssymb}
\usepackage{mathtools}
\usepackage{graphicx}
\usepackage{geometry}
\usepackage{a4wide}
\usepackage[latin1]{inputenc}
\usepackage{color}

\DeclareMathOperator{\SL}{SL}
\DeclareMathOperator{\Mp}{Mp}
\DeclareMathOperator{\Sp}{Sp}

\DeclareMathOperator{\ord}{ord}

\DeclareMathOperator{\N}{\mathbb{N}}

\DeclareMathOperator{\Z}{\mathbb{Z}}
\DeclareMathOperator{\R}{\mathbb{R}}
\DeclareMathOperator{\C}{\mathbb{C}}
\DeclareMathOperator{\Q}{\mathbb{Q}}
\renewcommand{\H}{\mathbb{H}}

\DeclareMathOperator{\tr}{tr}
\DeclareMathOperator{\real}{Re}
\DeclareMathOperator{\imag}{Im}

\DeclareMathOperator{\e}{\mathfrak{e}}
\DeclareMathOperator{\Gr}{Gr}

\usepackage{amsthm}
\newtheorem{Theorem}{Theorem}[section]
\newtheorem{Lemma}[Theorem]{Lemma}

\newtheorem{Corollary}[Theorem]{Corollary}

\theoremstyle{definition}

\theoremstyle{remark}

\numberwithin{equation}{section}

\begin{document} 

\title{Holomorphic Borcherds products of singular weight for simple lattices of arbitrary level}
\author{Sebastian Opitz and Markus Schwagenscheidt}
\date{}
\maketitle

\begin{abstract}
%	An even lattice $L$ of signature $(2,n)$ is called simple if the space of cusp forms of weight $\frac{n}{2}+1$ for the dual Weil representation of $L$ is trivial. For a simple lattice, every $\Z$-linear combination of Heegner divisors is the divisor of an automorphic product. Up to isomorphy there are only finitely many simple lattices, which have been determined by Bruinier, Ehlen and Freitag. If $n \geq 3$, then the smallest possible positive weight of a nontrivial holomorphic automorphic product, called the singular weight, is given by $\frac{n}{2}-1$. 
	We classify the holomorphic Borcherds products of singular weight for all simple lattices of signature $(2,n)$ with $n \geq 3$. In addition to the automorphic products of singular weight for the simple lattices of square free level found by Dittmann, Hagemeier and the second author, we obtain several automorphic products of singular weight $1/2$ for simple lattices of signature $(2,3)$. We interpret them as Siegel modular forms of genus $2$ and explicitly describe them in terms of the ten even theta constants. In order to rule out further holomorphic Borcherds products of singular weight, we derive estimates for the Fourier coefficients of vector valued Eisenstein series, which are of independent interest.
\end{abstract}

\section{Introduction}

In his celebrated work \cite{borcherds}, Borcherds defined a multiplicative lifting map from vector valued modular forms for the Weil representation associated to an even lattice $L$ to modular forms on the hermitian symmetric domain corresponding to $L$. The resulting modular forms have infinite product expansions at the cusps and are therefore called automorphic (or Borcherds) products. The smallest possible weight of a nonconstant holomorphic modular form for the orthogonal group of an even lattice $L$ of signature $(2,n)$ with $n\geq 3$, called the singular weight, is given by $\frac{n}{2}-1$, compare \cite{bundschuh}. Borcherds products of singular weight have interesting Fourier and product expansions which often yield denominator identities of generalized Kac-Moody algebras \cite{scheithauerclassification}. Furthermore, there are not many known holomorphic Borcherds products of singular weight, and it is a folklore conjecture that there are only finitely many of them, which makes it an interesting problem to classify them all. 

Scheithauer \cite{scheithauerclassification} obtained a complete list of the symmetric and reflective Borcherds products of singular weight for lattices of square free level. In \cite{scheithauerprimelevel}, he classified all reflective (not necessarily symmetric) holomorphic automorphic products of singular weight for lattices of prime level, and he gave an effective bound for the possible signatures of lattices of prime level (with prescribed discriminant group) allowing holomorphic Borcherds products of singular weight.
%, which roughly means that the corresponding vector valued modular forms are invariant under the action of the orthogonal group of the discriminant form associated to $L$ and that the pole orders of the vector valued modular forms are very small. 
Recently, Dittmann was able to remove the requirement of being symmetric for all lattices of square free level \cite{dittmann}. 

Following a somewhat different direction, Dittmann, Hagemeier and the second author in \cite{dittmannhagemeierschwagenscheidt} classified the simple lattices of square free level (hence even signature) and the corresponding holomorphic Borcherds products of singular weight. Here, an even lattice $L$ of signature $(2,n)$ is called simple if the space of cusp forms of weight $\frac{n}{2}+1$ for the dual Weil representation of $L$ vanishes. For a simple lattice, every formal principal part is the principal part of a vector valued modular form, which implies that a simple lattice allows many Borcherds products. One of the main result of \cite{dittmannhagemeierschwagenscheidt} is a list of $15$ (isomorphy classes of) simple lattices of square free level. It was further proven that only four of them admit holomorphic automorphic products of singular weight, which were then constructed explicitly. Shortly afterwards, Bruinier, Ehlen and Freitag \cite{bruinierehlenfreitag} determined all simple lattices of arbitrary level and signature. The main result of the present work is the classification of the holomorphic Borcherds products of singular weight for all simple lattices of signature $(2,n)$ with $n \geq 3$. To ensure that the Borcherds product is holomorphic, we assume that the corresponding vector valued modular form has only nonnegative coefficients in its principal part.

\begin{Theorem}\label{theorem classification}
	Holomorphic Borcherds products (coming from vector valued modular forms with nonnegative principal part) of singular weight $\frac{n}{2}-1$ for simple lattices $L$ of signature $(2,n), n \geq 3,$ only exist in the following cases.
	\begingroup
	\renewcommand{\arraystretch}{1.2}
	\begin{align*}
		\begin{array}{|c|c|l|c|}
		\hline 
		n & \text{genus} & \text{lattice} & \text{level} \\
		\hline \hline
		3 & 2_{7}^{+1}4^{+2} & A_{1}(-1) \oplus U \oplus U(4) & 4 \\
		  & 2_{7}^{+3}4^{+2} & A_{1}(-1)\oplus U(2) \oplus U(4) & 4 \\
		  & 2_{7}^{+1}4^{+4} & A_{1}(-1) \oplus U(4) \oplus U(4) & 4 \\
		  & 2^{+4}4_{7}^{+1} & A_{1}(-2) \oplus U(2) \oplus U(2) & 8 \\
		  & 8_{7}^{+1} & A_{1}(-4) \oplus U \oplus U & 16 \\
		\hline
		4 & 3^{+5} & A_{2}(-1) \oplus U(3) \oplus U(3) & 3 \\
		\hline
		6 & 2^{-6}& D_{4} \oplus U(2) \oplus U(2) & 2 \\
		\hline
		10 & 2^{+2}& E_{8}(-1) \oplus U \oplus U(2) & 2 \\
		\hline 
		26 & 1^{+1} & E_{8}(-1) \oplus E_{8}(-1) \oplus E_{8}(-1) \oplus U \oplus U & 1 \\
		\hline
		\end{array}
	\end{align*}
	\endgroup
	Here $U$ denotes the hyperbolic plane $\Z^{2}$ with $Q(x,y) = xy$, and $A_{1}, A_{2},D_{4},E_{8}$ denote the usual root lattices. Further, if $(L,Q)$ is a lattice and $N$ a positive integer, we write $L(N)$ for the rescaled lattice $(L,NQ)$.
\end{Theorem}

We remark that the automorphic products for the lattices with $n \geq 4$ in the above table were already found in \cite{dittmannhagemeierschwagenscheidt}. The automorphic products for the simple lattices of signature $(2,3)$ can be viewed as Siegel modular forms and can be written in terms of theta constants, see Section~\ref{section siegelmodularforms}.

We briefly explain the idea of the proof. Let $L$ be one of the simple even lattices of signature $(2,n)$ with $n \geq 3$ (which are given in the appendix). Let $L'$ denote its dual lattice and let $L'/L$ be its discriminant form. Let $f$ be a weakly holomorphic modular form of weight $1-\frac{n}{2}$ for the Weil representation of $L$ with real coefficients $a_{f}(\gamma,n)$ for all $\gamma \in L'/L, n \in \Q$. To make sure that the associated Borcherds product $\Psi_{f}$ is holomorphic we assume that the coefficients $a_{f}(\gamma,n)$ with $n < 0$, i.e., the coefficients of the principal part of $f$, are nonnegative integers.
The weight of $\Psi_{f}$ is given by the linear combination
\begin{align}\label{eq weight Borcherds product}
-\frac{1}{2}\sum_{\gamma \in L'/L}\sum_{n < 0}a_{f}(\gamma,n)a_{E}(\gamma,-n),
\end{align}
where $a_{E}(\gamma,-n)$ are the Fourier coefficients of an Eisenstein series of weight $\frac{n}{2}+1$ for the dual Weil representation (see Section~\ref{orthogonalmodularforms}). We have to check whether we can choose the coefficients $a_{f}(\gamma,n) \in \Z_{\geq 0}$ for $n < 0$ such that \eqref{eq weight Borcherds product} equals the singular weight $\frac{n}{2}-1$. To this end, we give an explicit lower bound for the absolute value of the coefficients of this Eisenstein series.

\begin{Theorem}\label{theorem estimates introduction}
	Let $L$ be an even lattice of signature $(b^{+},b^{-})$ ($b^{+}$ even) and rank $m \geq 3$, and let $d = |L'/L|$. Suppose that $L$ splits a rescaled hyperbolic plane $U(N)$. Let $\gamma \in L'/L$ and $n \in \Z-Q(\gamma)$ with $n > 0$. The coefficient $a_{E}(\gamma,n)$ of the Eisenstein series of weight $k = \frac{m}{2}$ for the dual Weil representation is either $0$ or satisfies the estimate
	\[
	(-1)^{b^{+}/2}a_{E}(\gamma,n) \geq C_{k,d,N}\cdot n^{k-1}
	\]
	for some explicit constant $C_{k,d,N} > 0$ depending on $k,d,$ and $N$, but not on $\gamma$ and $n$.
\end{Theorem}

We refer to Theorem~\ref{theorem Eisenstein estimates} for the details. Note that all the simple lattices of signature $(2,n)$ with $n \geq 3$ split a rescaled hyperbolic plane. The theorem implies that the weight of $\Psi_{f}$ is bigger than the singular weight for all but finitely many choices of the principal part of $f$. For the few remaining choices of the principal part of $f$, we explicitly compute the coefficients of the Eisenstein series to check whether \eqref{eq weight Borcherds product} equals the singular weight. The first author has written a python program which allows a very fast computation of the Fourier coefficients of Eisenstein series for the Weil representation, which will be part of his PhD thesis \cite{opitzthesis}. Finally, if we have found a suitable principal part such that \eqref{eq weight Borcherds product} equals the singular weight, the fact that $L$ is a simple lattice guarantees that it is the principal part of a weakly holomorphic modular form $f$, i.e., there exists a holomorphic Borcherds product of singular weight. 

The work is organized as follows. In Section~\ref{section preliminaries}, we start with the necessary preliminaries on vector valued modular forms for the Weil representation, Borcherds products and simple lattices. Then, in Section~\ref{section eisenstein}, we derive estimates for the Fourier coefficients of vector valued Eisenstein series and prove Theorem~\ref{theorem estimates introduction}. Section~\ref{section proof theorem classification} is devoted to the proof of Theorem~\ref{theorem classification}. Finally, in Section~\ref{section siegelmodularforms}, we interpret the holomorphic Borcherds products of singular weight for the simple lattices of signature $(2,3)$ as Siegel modular forms of genus $2$ and describe them in terms of theta constants. In the appendix, we recall the list of simple lattices of signature $(2,n), n \geq 3,$ of Bruinier, Ehlen and Freitag \cite{bruinierehlenfreitag}.

\section{Preliminaries}\label{section preliminaries}

\subsection{Vector valued modular forms} We let $L$ be an even lattice of signature $(b^{+},b^{-})$ with quadratic form $Q$ and bilinear form $(\cdot,\cdot)$, and we let $\C[L'/L]$ be the group algebra of $L'/L$ with basis vectors $\e_{\gamma}$ for $\gamma \in L'/L$. By $\Mp_{2}(\Z)$ we denote the integral metaplectic group consisting of pairs $(M, \phi)$, where $M = \left(\begin{smallmatrix}a & b \\ c & d \end{smallmatrix} \right) \in \SL_2(\Z)$ and $\phi:\H\rightarrow \C$ is a holomorphic function with $\phi(\tau)^{2}=c\tau+d$. The Weil representation $\rho_{L}$ of $\Mp_{2}(\Z)$ is defined for the generators $S=\left(\left(\begin{smallmatrix}0 & -1 \\ 1 & 0 \end{smallmatrix} \right),\sqrt{\tau}\right)$ and $T=\left(\left(\begin{smallmatrix}1 & 1 \\ 0 & 1 \end{smallmatrix} \right), 1\right)$ of $\Mp_2(\Z)$ and $\gamma \in L'/L$ by the formulas
\begin{align*}
 \rho_{L}(T) \e_{\gamma} &= e(Q(\gamma)) \e_\gamma,\qquad  \rho_{L}(S) \e_{\gamma} = \frac{e((b^{-}-b^{+})/8)}{\sqrt{|L'/L|}}\sum_{\mu \in L'/L} e(-(\gamma,\mu)) \e_{\mu},
\end{align*}
where $e(a):=e^{2\pi i a}$ for $a \in \C$. The dual Weil representation will be denoted by $\rho_{L}^{*}$. 

A holomorphic function $f: \H \to \C[L'/L]$ is called a (weakly holomorphic) modular form of weight $k \in \frac{1}{2}\Z$ for $\rho_{L}$ if $f(Mz) = \phi(z)^{2k}\rho_{L}(M,\phi)f(z)$ for all $(M,\phi) \in \Mp_{2}(\Z)$ and $z \in \H$, and if $f$ has at most a pole at $\infty$. More precisely, this means that $f$ has a Fourier expansion of the form
\[
f(z) = \sum_{\gamma \in L'/L}\sum_{\substack{n \in \Q \\ n \gg -\infty}}a_{f}(\gamma,n)e(nz)\e_{\gamma},
\]
that is, $a_{f}(\gamma,n) = 0$ for all but finitely many $n < 0$. 
The finite Fourier polynomial 
\begin{align}\label{eq principal part}
P_{f}(z) = \sum_{\gamma \in L'/L}\sum_{n < 0}a_{f}(\gamma,n)e(nz)\e_{\gamma}
\end{align}
is called the principal part of $f$. Note that the coefficients of $f$ satisfy the symmetry  $a_{f}(\gamma,n) = (-1)^{k-(b^{+}-b^{-})/2}a_{f}(-\gamma,n)$ for all $n \in \Q, \gamma \in L'/L$. Every finite sum as in \eqref{eq principal part} which satisfies this symmetry will be called a formal principal part.

\subsection{Orthogonal modular forms and Borcherds products}\label{orthogonalmodularforms}

Let $L$ be an even lattice of signature $(2,n)$ with $n\geq 3$ and let $V = L \otimes \R$. We let $\Gr(V)$ be the Grassmannian of positive definite planes in $V$. Choose some primitive isotropic vector $z \in L$ and some vector $z' \in L'$ with $(z,z') = 1$, and let $K = L \cap z^{\perp} \cap z'^{\perp}$. The complex manifold $\{Z = X + iY \in K \otimes \C: (Y,Y) > 0\}$ has two connected components. We pick one of them and denote it by $\mathcal{H}_{n}$. It can be viewed as a generalized upper half-plane. There is a bijection $\Gr(V) \cong \mathcal{H}_{n}$ which endows $\Gr(V)$ with a complex structure, compare \cite{bruinierhabil}, Section~3.2.

We let $O(L)^{+} = O(L) \cap O(V)^{+}$ be the intersection of the orthogonal group $O(L)$ of $L$ with the identity component of $O(V)$, and we let $\Gamma_{L}$ be the kernel of the natural map $O(L)^{+} \to O(L'/L)$. It has finite index in $O(L)^{+}$. The action of $O(L)^{+}$ on $\Gr(V)$ induces an action on $\mathcal{H}_{n}$. Further, there is a natural factor of automorphy $j(\sigma,Z)$ for $\sigma \in O(L)^{+}$ and $Z \in \mathcal{H}_{n}$, see \cite{bruinierhabil}, Section~3.3. A meromorphic function $\Psi: \mathcal{H}_{n} \to \C$ is called a modular form of weight $k \in \frac{1}{2}\Z$ for $\Gamma_{L}$ and multiplier system $\chi$ if $\Psi(\sigma Z) = \chi(\sigma)j(\sigma,Z)^{k}\Psi(Z)$ for all $\sigma \in \Gamma_{L}$ and $Z \in \mathcal{H}_{n}$. The smallest possible positive weight of a nontrivial holomorphic modular form for $\Gamma_{L}$ is called the singular weight. It is given by $\frac{n}{2}-1$, compare \cite{bundschuh}.

For $\gamma \in L'/L$ and $n < 0$ we define the Heegner divisor of index $(\gamma,n)$ by
\[
H_{L}(\gamma,n) = \sum_{\substack{X \in L + \gamma \\ Q(X) = n}}X^{\perp} \subset \Gr(V).
\]
Here $X^{\perp} \subset \Gr(V)$ denotes the set of all positive definite planes orthogonal to $X$. The corresponding divisor in $\mathcal{H}_{n}$ will be denoted by the same symbol $H_{L}(\gamma,n)$.

\begin{Theorem}[\cite{borcherds}, Theorem~13.3]
	Let $f = \sum_{\gamma,n}a_{f}(\gamma,n)e(nz)\e_{\gamma}$ be a weakly holomorphic modular form of weight $1-\frac{n}{2}$ for $\rho_{L}$ with $a_{f}(\gamma,n) \in \Z$ for all $n \leq 0,\gamma \in L'/L$. Then there exists a meromorphic modular form $\Psi_{f}: \mathcal{H}_{n} \to \C$ of weight $a_{f}(0,0)/2$ for $\Gamma_{L}$ (transforming with a multiplier system of finite order) whose divisor is given by
		\[
		\frac{1}{2}\sum_{\gamma \in L'/L}\sum_{n < 0}a_{f}(\gamma,n)H_{L}(\gamma,n).
		\]
		Here $H_{L}(\gamma,n)$ has multiplicity $2$ if $2\gamma = 0$ in $L'/L$, and multiplicity $1$ otherwise.
\end{Theorem}

The modular form $\Psi_{f}$ has particular product expansions at the cusps, which is why it is called the Borcherds product or automorphic product associated to $f$. We did not include the product expansions here since we will not use them.

We are particularly interested in Borcherds products of singular weight $\frac{n}{2}-1$. Therefore, we need to control the constant coefficient $a_{f}(0,0)$ of $f$. Let $\kappa = \frac{n}{2}+1$, and define a vector valued Eisenstein series for the dual Weil representation $\rho_{L}^{*}$ by
\[
E(z) = \frac{1}{4}\sum_{(M,\phi) \in \tilde{\Gamma}_{\infty} \backslash \Mp_{2}(\Z)}\phi(z)^{-2\kappa}\rho_{L}^{*}(M,\phi)^{-1}\e_{0},
\]
where $\tilde{\Gamma}_{\infty}$ is the subgroup of $\Mp_{2}(\Z)$ generated by $T$. It is a modular form of weight $\kappa$ for $\rho_{L}^{*}$ and it has a Fourier expansion of the form
\[
E(z) = \e_{0} + \sum_{\gamma \in L'/L}\sum_{n > 0}a_{E}(\gamma,n)e(nz)\e_{\gamma},
\]
with Fourier coeffficients $a_{E}(\gamma,n) \in \Q$, compare Theorem~\ref{thm:bruinier_kuss} below. If $f$ is a weakly holomorphic modular form of weight $k = 1-\frac{n}{2} = 2-\kappa$, then the function $\sum_{\gamma \in L'/L}f_{\gamma}(z)E_{\gamma}(z)dz$ is a meromorphic $1$-form on $\SL_{2}(\Z)\backslash \H$. By the residue theorem its residue vanishes, which yields the formula
\[
\text{weight of }\Psi_{f} = \frac{1}{2}a_{f}(0,0) = -\frac{1}{2}\sum_{\gamma \in L'/L}\sum_{n < 0}a_{f}(\gamma,n)a_{E}(\gamma,-n).
\]
Therefore, the constant coefficient of $f$, and hence the weight of the associated Borcherds product, is determined by the principal part of $f$ and coefficients of an Eisenstein series.

\subsection{Simple lattices} 
An even lattice $L$ of signature $(2,n)$ is called simple if the space of cusp forms of weight $\frac{n}{2}+1$ for $\rho_{L}^{*}$ is trivial. This space of cusp forms is also called the obstruction space for $L$. The significance of this notion is the fact that a formal principal part as in \eqref{eq principal part} is the principal part of a weakly holomorphic modular form of weight $1-\frac{n}{2}$ for $\rho_{L}$ if and only if
\[
\sum_{\gamma \in L'/L}\sum_{n < 0}a_{f}(\gamma,n)a_{g}(\gamma,-n) = 0
\]
for every cusp form $g(z) = \sum_{\gamma,n}a_{g}(\gamma,n)e(nz)\e_{\gamma}$ in the obstruction space. Hence, for a simple lattice every formal principal part is the principal part of a weakly holomorphic modular form of weight $1-\frac{n}{2}$, and hence every $\Z$-linear combination of Heegner divisors is the divisor of a Borcherds product. 

The simple lattices of square free level and the corresponding holomorphic Borcherds products of singular weight were determined in \cite{dittmannhagemeierschwagenscheidt}. Later, all simple lattices of arbitrary level were computed in \cite{bruinierehlenfreitag}, but the corresponding Borcherds products of singular weight were not studied. For convenience of the reader, we give a list of the simple lattices of signature $(2,n)$ with $n\geq 3$ in the appendix. 

\section{Fourier coefficients of Eisenstein series}\label{section eisenstein}

In this section we derive estimates for the coefficients of the Eisenstein series $E(z)$ and prove Theorem~\ref{theorem estimates introduction}. The coefficients of $E(z)$ are given by the following

\begin{Theorem}[\cite{BruinierKuss2001}, Theorem 4.8]\label{thm:bruinier_kuss}\label{theorem bruinierkuss}
Let $\gamma\in L'$ and $n\in \Z - q(\gamma)$ with $n>0$. The coefficient $a_{E}(\gamma, n)$ of the Eisenstein series $E(z)$ of weight $k=m/2$ for $\rho_{L}^{*}$ is equal to
\[
\frac{2^{k+1}\pi^{k}n^{k-1}(-1)^{b_{+}/2}}{\sqrt{|L'/L|}\Gamma(k)}
\]
times
\[
\begin{dcases}
\frac{\sigma_{1-k}(\tilde n, \chi_{4D})}{L(k, \chi_{4D})}\prod\limits_{p\mid 2\det(S)} p^{w_p(1-2k)}N_{\gamma, n}^{L}(p^{w_p}), & \text{if }2\mid m,\\
\frac{L(k-1/2, \chi_{\mathcal D})}{\zeta(2k-1)}\sum\limits_{d\mid f}\mu(d)\chi_{\mathcal D}(d)d^{1/2-k}\sigma_{2-2k}(f/d)\prod\limits_{p\mid 2\det(S)} \frac{p^{w_p(1-2k)}N_{\gamma, n}^{L}(p^{w_p})}{1-p^{1-2k}}, & \text{if }2\nmid m.
\end{dcases}
\]
Here $S$ is the Gram matrix of $L$ and 
\[
N_{\gamma,n}^{L}(a) = \#\{r \in L/aL: Q(r - \gamma) + n \equiv 0 \mod a\}
\]
is a representation number. Furthermore, for a Dirichlet character $\chi$, $\sigma_{s}(m,\chi) = \sum_{d \mid m}\chi(d)d^{s}$ is a divisor sum twisted by $\chi$ and $L(s,\chi)$ is the usual Dirichlet $L$-function. The definition of $w_p$, $D$, $\mathcal D$, $f$, $\tilde n$, $\chi_{4D}$, and $\chi_{\mathcal{D}}$ can be found in the reference.
\end{Theorem}

If the lattice in question splits a rescaled hyperbolic plane, we may estimate the coefficients as follows.

\begin{Theorem}\label{thm:eis_lower_bound}\label{theorem Eisenstein estimates}
Let $L$ be a lattice of signature $(b_{+},b_{-})$ ($b_{+}$ even) with rank $m \ge 3$ such that $L = L_{1} \oplus U(N)$ for some even lattice $L_{1}$ of rank $m-2$. Let $d = |L'/L|$. Let $\gamma\in L'/L$ and $n\in \Z - Q(\gamma)$ with $n>0$. The coefficient $a_{E}(\gamma, n)$ of the Eisenstein series $E(z)$ of weight $k = m/2$ for $\rho_{L}^{*}$ is either $0$ or
\[
(-1)^{b_{+}/2} a_{E}(\gamma, n) \ge C_{k, d, N}\cdot n^{k-1},
\]
where $C_{k, d, N}$ is given by
\[
\frac{2^{k+1}\pi^{k}}{\sqrt{d}\,\Gamma(k)}\times\begin{dcases}
\frac{2-\zeta(k-1)}{\zeta(k)}\prod\limits_{p\mid 2d} p^{(3-2k)\ord_{p}(N)}(1-1/p), & \text{if }2\mid m,\\
\frac{2-\zeta(k-1/2)}{\zeta(k-1/2)}\prod\limits_{p\mid 2d} \frac{p^{(3-2k)\ord_{p}(N)}(1-1/p)}{1-p^{1-2k}}, & \text{if }2\nmid m.
\end{dcases}
\]
\end{Theorem}

The proof is accomplished using the following lemmas, which generalize the estimates in \cite{BruinierMoeller2017}. We start with a well-known formula for the representation numbers of the hyperbolic plane.

\begin{Lemma}\label{lemma:rep_num_hyp_exact}Let $n\in\Z$ and $\nu\in\Z_{\ge0}$. Then
\[
N^{U(1)}_{0,n}(p^\nu)=
\begin{cases}
(\ord_{p}(n)+1)(1-1/p)p^{\nu}, & \text{if }\ord_{p}(n) < \nu,\\
\nu(1-1/p)p^{\nu}+p^\nu, & \text{if }\ord_{p}(n) \ge \nu.
\end{cases}
\]
\end{Lemma}

\begin{Corollary}\label{cor:rep_num_hyp_estimate}We have
\[
(1-1/p) \le p^{-\nu}N^{U(1)}_{0,n}(p^\nu) \le \nu +1.
\]
\end{Corollary}

Next, we determine the representation numbers of rescaled hyperbolic planes.

\begin{Lemma}\label{lemma:rep_num_scaled_hyp_exact}Let $p$ be a prime, $\nu\in\N_{0}$, $N\in\Z$ and $\gamma = \left(\frac{\gamma_1}{N},\frac{\gamma_2}{N}\right)\in U(N)'=\frac{1}{N}\Z^2$.
We write $p^{\nu_N}\|N$, $p^{\nu_\gamma}\|(\gamma_1,\gamma_2)$ ($\nu_\gamma = \infty$ for $\gamma=(0,0)$) and $n=\ell-\frac{\gamma_1\gamma_2}{N}$ with $\ell\in\Z$.
Furthermore, we define $\nu_{\min}=\min(\nu,\nu_\gamma,\nu_N)$. Then
\[
N^{U(N)}_{\gamma,n}(p^\nu)=
\begin{cases}
0, & \text{if } p^{\nu_{\min}} \nmid \ell,\\
p^{2\nu_N}N^{U(1)}_{0,\tilde n}(p^{\nu-\nu_N}), &  \text{if }\nu_N \le \min(\nu,\nu_\gamma)  \text{ and } p^{\nu_{\min}} \mid \ell,\\ 
p^{\nu+\min(\nu,\nu_\gamma)}, & \text{if } \nu_N > \min(\nu,\nu_\gamma) \text{ and } p^{\nu_{\min}} \mid \ell,
\end{cases}
\]
where $\tilde n = Nnp^{-2\nu_N}$.
\end{Lemma}

\begin{proof}
We may write
\begin{align*}
N^{U(N)}_{\gamma,n}(p^\nu) 
&= \#\{(a,b) \in (\Z/p^{\nu}\Z)^{2} : Nab -(a\gamma_2 + b\gamma_1) \equiv -\ell \pmod{p^{\nu}}\}.
\end{align*}

If $p^{\nu_{\min}} \nmid \ell$, then the condition for $(a,b)$ implies $0\equiv -l\not\equiv 0\ \pmod{p^{\nu_{\min}}}$. This condition cannot be fulfilled and the representation number is $0$ in this case.

If $\nu_N\le \min(\nu,\nu_\gamma)$, we write $N=N'\cdot p^{\nu_N}$ with $(N',p) = 1$ and find an integer $\overline N'$ such that $\overline N' \equiv (N')^{-1} \pmod{p^{\nu}}$.
The bijection $a_1 \mapsto \overline N'(a_1 + \frac{\gamma_1}{p^{\nu_{N}}}) = a$ and $b_1 \mapsto \overline N'(b_1 + \frac{\gamma_2}{p^{\nu_{N}}}) = b$
shows that
\begin{align*}
N^{U(N)}_{\gamma,n}(p^\nu) &= \#\left\{(a_1,b_1) \in (\Z/p^{\nu}\Z)^{2} : a_1 b_1\equiv -Np^{-2\nu_N}n \pmod{p^{\nu-\nu_N}}\right\}\\
&=p^{2\nu_N}N^{U(1)}_{0,\tilde n}(p^{\nu-\nu_N}).
\end{align*}
This proves the second case.

If $\nu_N> \min(\nu,\nu_\gamma)$, we distinguish two cases.
If $\nu_\gamma \ge \nu$, the condition for $(a,b)$ is trivial and we have $p^{2\nu}=p^{\nu+\min(\nu,\nu_\gamma)}$ solutions. If $\nu_\gamma < \nu$, we may assume that $p^{\nu_\gamma}\| \gamma_2$. We see that
\begin{align*}
N^{U(N)}_{\gamma,n}(p^\nu) &= \#\left\{(a,b) \in (\Z/p^{\nu}\Z)^{2} : a \equiv \left(\frac{N}{p^{\nu_\gamma}}b-\frac{\gamma_2}{p^{\nu_\gamma}}\right)^{-1}\frac{b\gamma_1-\ell}{p^{\nu_\gamma}} \pmod{p^{\nu-\nu_\gamma}}\right\} = p^{\nu+\nu_\gamma},
\end{align*}
so again we have $p^{\nu+\min(\nu,\nu_\gamma)}$ solutions.
\end{proof}

Using the above lemma, we derive a lower bound for the representation numbers of lattices which split a rescaled hyperbolic plane.

\begin{Lemma}
    Let $L$ be a lattice of rank $m \ge 3$ such that $L = L_{1} \oplus U(N)$ for some even lattice $L_{1}$ of rank $m-2$. Then either $N_{\gamma,n}^{L}(p^{\nu}) = 0$ or
    \[
    p^{\nu(1-2k)}N_{\gamma,n}^{L}(p^{\nu}) \ge p^{(3-2k)\nu_N}(1-1/p).
    \]
\end{Lemma}

\begin{proof}
    Write $\gamma = \gamma_{1}+\gamma_{2}$ with $\gamma_{1} \in L_{1}'$ and $\gamma_{2} \in U(N)'$. We may write
    \[
    N_{\gamma,n}^{L}(p^{\nu}) = \sum_{\lambda_{1} \in L_{1}/p^{\nu}L_{1}}N_{\gamma_{2},n+Q(\lambda_{1}-\gamma_{1})}^{U(N)}(p^{\nu}).
    \]
    To estimate the summands we define $\nu_\gamma:=\nu_{\gamma_2}$, $\nu_N$ and $\nu_{\min}$ as in Lemma \ref{lemma:rep_num_scaled_hyp_exact}.
    If all summands are $0$, there is nothing to prove. Therefore, we may assume that there is a $\lambda_1$ such that the corresponding summand $N_{\gamma_{2},n+Q(\lambda_{1}-\gamma_{1})}^{U(N)}(p^{\nu})$ is nonzero.
    This implies
    \[
    p^{\nu_{\min}}\mid \ell = n + Q(\lambda_1 - \gamma_1) + Q(\gamma_2).
    \]
    If we change $\lambda_1$ modulo $p^{\nu_{\min}}L_1$, this remains true.
    This gives at least $p^{(\nu-\nu_{\min})(m-2)}$ nonzero summands, which we can estimate using Lemma \ref{lemma:rep_num_scaled_hyp_exact}.
    
    We distinguish the cases $\nu_N \le \min(\nu, \nu_\gamma)$ and $\nu_N > \min(\nu, \nu_\gamma)$. In the first case, the nonzero summands are of the form
    \[
    p^{2\nu_N}N_{0,\tilde n}^{U(1)}(p^{\nu-\nu_N}) \ge p^{2\nu_N}p^{\nu-\nu_N}(1-1/p)
    \]
    where $\tilde n$ might depend on $\lambda_1$ and we use Corollary \ref{cor:rep_num_hyp_estimate} for the estimate.
    This yields
    \[
    N_{\gamma,n}^{L}(p^{\nu}) \ge p^{(\nu-\nu_{\min})(m-2)}p^{2\nu_N}p^{\nu-\nu_N}(1-1/p)=p^{(m-1)\nu}p^{(3-m)\nu_N}(1-1/p)
    \]
    for the sum. In the second case, the nonzero summands are of the form $p^{\nu+\min(\nu, \nu_\gamma)} = p^{\nu+\nu_{\min}}$. This yields
    \[
    N_{\gamma,n}^{L}(p^{\nu}) \ge p^{(\nu-\nu_{\min})(m-2)}p^{\nu+\nu_{\min}}=p^{(m-1)\nu}p^{(3-m)\nu_{\min}} \ge p^{(m-1)\nu}p^{(3-m)\nu_N}(1-1/p),
    \]
    where we have used $3-m\le 0$ and $\nu_{\min}\le \nu_N$. The proof is finished.
\end{proof}

Note that the characters $\chi_{4D}$ and $\chi_{\mathcal{D}}$ appearing in the Fourier expansion of $E(z)$ given in Theorem~\ref{theorem bruinierkuss} are quadratic Dirichlet characters.
For even signature, we need the following estimate.

\begin{Lemma}
    Let $\chi$ be a real Dirichlet character, $n \in \N$, and $s \in \R$ with $s \geq 2$. Then
    \[
    \zeta(s) \geq \sigma_{-s}(n,\chi) \geq 2-\zeta(s).
    \]
\end{Lemma}

\begin{proof}
    We have
    \begin{align*}
    \zeta(s) = \sum_{d\geq 1}d^{-s} \geq \sum_{d \mid n}\chi(d)d^{-s} \geq 2-\sum_{d\mid n}d^{-s} \geq 2-\sum_{d \geq 1}d^{-s} =  2-\zeta(s),
    \end{align*}
    which yields the desired estimates.
\end{proof}

For odd signature, the following two estimates are useful.

\begin{Lemma}
    Let $\chi$ be a real Dirichlet character, let $f \in \N$, and let $k \geq 5/2$. Then we have
    \[
    \sum_{d \mid f}\mu(d)\chi(d)d^{1/2-k}\sigma_{2-2k}(f/d) > 2-\zeta(k-1/2).
    \]
\end{Lemma}

\begin{proof}
    We split off the term for $d = 1$ on the left-hand side and estimate
    \begin{align*}
    \sum_{d \mid f}\mu(d)\chi(d)d^{1/2-k}\sigma_{2-2k}(f/d) &=\sigma_{2-2k}(f)+\sum_{\substack{d \mid f \\ d \neq 1}}\mu(d)\chi(d)d^{1/2-k}\sigma_{2-2k}(f/d) \\
    &\geq 2\sigma_{2-2k}(f)-\sum_{\substack{d \mid f }}d^{1/2-k}\sigma_{2-2k}(f/d).
    \end{align*}
    Now $\sigma_{2-2k}(f/d) \leq \sigma_{2-2k}(f)$ for $d \mid f$, so the last expression is greater or equal than
    \begin{align*}
    \sigma_{2-2k}(f)\bigg(2-\sum_{\substack{d \mid f }}d^{1/2-k}\bigg) > \sigma_{2-2k}(f)\left(2-\zeta(k-1/2)\right) \geq 2-\zeta(k-1/2).
    \end{align*}
    This finishes the proof.
\end{proof}

\begin{Lemma}
    Let $\chi$ be a real Dirichlet character and let $s \in \R, s > 1$. Then
    \[
    \zeta(s) \geq L(s,\chi) \geq \frac{\zeta(2s)}{\zeta(s)}.
    \]
\end{Lemma}

\begin{proof}
    For $s > 1$ we have $L(s,\chi) = \prod_{p}(1-\chi(p)p^{-s})^{-1}$ and
    \[
   \zeta(s) = \prod_{p}\frac{1}{1-p^{-s}} \geq \prod_{p}\frac{1}{1-\chi(p)p^{-s}} \geq \prod_{p}\frac{1}{1+p^{-s}} = \frac{\zeta(2s)}{\zeta(s)},
    \]
  	which completes the proof.
\end{proof}

Putting together all the above lemmas, we easily obtain the estimates in Theorem \ref{thm:eis_lower_bound}.

\section{The proof of Theorem~\ref{theorem classification}}\label{section proof theorem classification}

For a given lattice of signature $(2,n)$ with $n\geq 3$ we are interested in solutions to the equation
\[
\frac{n}{2} -1  \stackrel{!}{=} -\frac{1}{2}\sum_{\gamma \in L'/L}\sum_{n < 0}a_{f}(\gamma,n)a_{E}(\gamma,-n)
\]
with $a_{f}(\gamma,n)\in\Z_{\ge0}$ satisfying $a_{f}(\gamma,n)=a_{f}(-\gamma,n)$. We know from Theorem~\ref{thm:bruinier_kuss} that the
Eisenstein coefficients on the right hand side are nonpositive. For any $\gamma\in L'/L$ and any $n < 0$ we have $a_{E}(\gamma,-n)= a_{E}(-\gamma,-n)$. Hence any nonzero summand for a $\gamma$ of order greater than $2$ will occur twice, once for $\gamma$ and again for $-\gamma \ne \gamma$. We can only find a solution to the above equation if there is an Eisenstein coefficient satisfying $2-n \le a_{E}(\gamma,-n) < 0$ and $2\gamma=0$
or an Eisenstein coefficient satisfying $1-\frac{n}{2} \le a_{E}(\gamma,-n) < 0$ and $2\gamma\ne 0$. 

We consider one of the simple lattice given in the appendix in detail. The other lattices can be treated analogously. The lattice $L=A_1(-1)\oplus U(4) \oplus U(4)$ has genus symbol $2_7^{+1}4^{+4}$ and signature $(2,3)$. We need to check for Eisenstein coefficients satisfying $-1 \le a_{E}(\gamma,-n) < 0$ and $2\gamma=0$
or Eisenstein coefficients satisfying $-\frac{1}{2} \le a_{E}(\gamma,-n) < 0$ and $2\gamma\ne 0$.
The lattice splits a hyperbolic plane rescaled by $4$, which leads to the estimate
\[
-a_{E}(\gamma, n) \ge C_{ 4 , 512 , \frac{5}{2} }\cdot n^{\frac{3}{2}}
\]
for the nonzero coefficients of the Eisenstein series of weight $\frac{1}{2}$ for the dual Weil representation. We have
\[
C_{ 4 , 512 , \frac{5}{2} }= -\frac{1}{90} \, \pi^{2} + \frac{2}{15} \approx 0.023671\text{.}
\]
For
\[
n\ge\left[C_{ 4 , 512 , \frac{5}{2} }^{ -\frac{2}{3} }\right]+1 = 13
\]
this implies $a_{E}(\gamma, n) < -1$.

Let $f$ be a modular form of weight $-\frac{1}{2}$ for $\rho_{L}$ with coefficients $a_{f}(\gamma,n)$.
In view of the discussion above and formula \eqref{eq weight Borcherds product} for the weight of the Borcherds products $\Psi_{f}$,
we see that the weight of $\Psi_{f}$ will be bigger than the singular weight $\frac{1}{2}$ if $a_{f}(\gamma,n) > 0$ for some $n \leq -13$.
Hence it suffices to compute the Eisenstein coefficients $a_{E}(\gamma,n)$ for $n < 13$.
The discriminant form of $L$ is isomorphic to $\Z/2\Z \times (\Z/4\Z)^{4}$ and has $8$ orbits with respect to the action of its orthogonal group.
Since the Eisenstein series is invariant under the orthogonal group it suffices to list the coefficients of the Eisenstein series once for each orbit.
The computation based on \cite{KudlaYang2010} is implemented in a python program using sage. The program for the computation of the Eisenstein coefficients will be part of the PhD thesis \cite{opitzthesis} of the first author and will be available on github.
The following table gives a representative for each orbit, the size of the orbit and the coefficients of the Eisenstein series for an element in this orbit.
\begingroup
\renewcommand{\arraystretch}{1.4}
\begin{center}
\begin{tabular}{l|l|l}
orbit repr. & $\#$orbit & $q$-expansion\\\hline
$(0, 0, 0, 0, 0)$ & $1$ & $\begin{array}{rl}& 1 -10 \,q^{ 1 } -70 \,q^{ 4 } -48 \,q^{ 5 } -120 \,q^{ 8 } -250 \,q^{ 9 } \\ &~ -240 \,q^{ 12 } +\operatorname{O}(\,q^{ 13 })\end{array}$\\\hline
$(0, 1, 0, 0, 0)$ & $120$ & $\begin{array}{rl}&  -4 \,q^{ 1 } -8 \,q^{ 2 } -16 \,q^{ 3 } -32 \,q^{ 4 } -32 \,q^{ 5 } -48 \,q^{ 6 }  \\ &~ -64 \,q^{ 7 } -64 \,q^{ 8 } -100 \,q^{ 9 } -112 \,q^{ 10 } -112 \,q^{ 11 }  \\ &~ -128 \,q^{ 12 } +\operatorname{O}(\,q^{ 13 })\end{array}$\\\hline
$(0, 2, 0, 0, 0)$ & $15$ & $\begin{array}{rl}&  -4 \,q^{ 1 } -8 \,q^{ 2 } -16 \,q^{ 3 } -32 \,q^{ 4 } -32 \,q^{ 5 } -48 \,q^{ 6 }  \\ &~ -64 \,q^{ 7 } -64 \,q^{ 8 } -100 \,q^{ 9 } -112 \,q^{ 10 } -112 \,q^{ 11 }  \\ &~ -128 \,q^{ 12 } +\operatorname{O}(\,q^{ 13 })\end{array}$\\\hline
$(0, 1, 1, 0, 0)$ & $120$ & $\begin{array}{rl}&  -2 \,q^{ \frac{3}{4} } -8 \,q^{ \frac{7}{4} } -14 \,q^{ \frac{11}{4} } -24 \,q^{ \frac{15}{4} } -38 \,q^{ \frac{19}{4} }  \\ &~ -40 \,q^{ \frac{23}{4} } -56 \,q^{ \frac{27}{4} } -80 \,q^{ \frac{31}{4} } -76 \,q^{ \frac{35}{4} } -104 \,q^{ \frac{39}{4} }  \\ &~ -126 \,q^{ \frac{43}{4} } -112 \,q^{ \frac{47}{4} } -156 \,q^{ \frac{51}{4} } +\operatorname{O}(\,q^{ \frac{55}{4} })\end{array}$\\\hline
$(0, 1, 2, 0, 0)$ & $120$ & $\begin{array}{rl}&  -1 \,q^{ \frac{1}{2} } -6 \,q^{ \frac{3}{2} } -14 \,q^{ \frac{5}{2} } -20 \,q^{ \frac{7}{2} } -31 \,q^{ \frac{9}{2} }  \\ &~ -46 \,q^{ \frac{11}{2} } -50 \,q^{ \frac{13}{2} } -68 \,q^{ \frac{15}{2} } -92 \,q^{ \frac{17}{2} } -82 \,q^{ \frac{19}{2} }  \\ &~ -108 \,q^{ \frac{21}{2} } -148 \,q^{ \frac{23}{2} } -131 \,q^{ \frac{25}{2} } +\operatorname{O}(\,q^{ \frac{27}{2} })\end{array}$\\\hline
$(1, 0, 0, 1, 0)$ & $120$ & $\begin{array}{rl}&  -\frac{1}{2} \,q^{ \frac{1}{4} } -4 \,q^{ \frac{5}{4} } -\frac{25}{2} \,q^{ \frac{9}{4} } -20 \,q^{ \frac{13}{4} } -24 \,q^{ \frac{17}{4} }  \\ &~ -40 \,q^{ \frac{21}{4} } -\frac{121}{2} \,q^{ \frac{25}{4} } -60 \,q^{ \frac{29}{4} } -72 \,q^{ \frac{33}{4} } -100 \,q^{ \frac{37}{4} }  \\ &~ -96 \,q^{ \frac{41}{4} } -124 \,q^{ \frac{45}{4} } -\frac{337}{2} \,q^{ \frac{49}{4} } +\operatorname{O}(\,q^{ \frac{53}{4} })\end{array}$\\\hline
$(1, 0, 0, 0, 0)$ & $10$ & $\begin{array}{rl}&  -1 \,q^{ \frac{1}{4} } -25 \,q^{ \frac{9}{4} } -48 \,q^{ \frac{17}{4} } -121 \,q^{ \frac{25}{4} } -144 \,q^{ \frac{33}{4} }  \\ &~ -192 \,q^{ \frac{41}{4} } -337 \,q^{ \frac{49}{4} } +\operatorname{O}(\,q^{ \frac{53}{4} })\end{array}$\\\hline
$(1, 2, 2, 0, 0)$ & $6$ & $\begin{array}{rl}&  -8 \,q^{ \frac{5}{4} } -40 \,q^{ \frac{13}{4} } -80 \,q^{ \frac{21}{4} } -120 \,q^{ \frac{29}{4} }  \\ &~ -200 \,q^{ \frac{37}{4} } -248 \,q^{ \frac{45}{4} } +\operatorname{O}(\,q^{ \frac{49}{4} })\end{array}$\\
\end{tabular} 
\end{center}
\endgroup

We see that there are exactly two possibilities to obtain holomorphic Borcherds products of singular weight, namely by setting 
\[
a_{f}(\gamma,\tfrac{1}{4}) = a_{f}(-\gamma,\tfrac{1}{4}) = 1
\]
for any $\gamma$ in the 6th orbit or by setting 
\[
a_{f}(\gamma,\tfrac{1}{4}) = 1
\]
for any $\gamma$ in the 7th orbit
(and $a_{f}(\gamma,n) = 0$ for all other $\gamma \in L'/L , n < 0$). We will call such elements $\gamma$ (which lead to Borcherds products of singular weight) \emph{good elements}. This finishes the classification of the holomorphic Borcherds products of singular weight for the simple lattice $L=A_1(-1)\oplus U(4) \oplus U(4)$. The other simple lattices can be treated analogously.

\section{Automorphic products as Siegel modular forms}\label{section siegelmodularforms}

We now describe the automorphic products for the simple lattices of signature $(2,3)$ as Siegel modular forms. To this end, we first recall the well-known identification of the Siegel upper half-space of genus 2 with the hermitian symmetric space associated to $O(2,3)$. We use the setup of \cite{lippolt}.

We consider the real quadratic space
\[
V = \left\{ \begin{pmatrix}x_{5} & -x_{3} & 0 & -x_{1} \\ x_{4} & -x_{5} & x_{1} & 0 \\ 0 & -x_{2} & x_{5} & x_{4} \\ x_{2} & 0 & -x_{3} & -x_{5}\end{pmatrix}: x_{i} \in \R\right\}, \qquad Q(X) = -\frac{1}{4}\tr(X^{2}) = x_{1}x_{2} + x_{3}x_{4} - x_{5}^{2}.
\]
It has signature $(2,3)$. Occasionally, we identify $V$ with $\R^{5}$ and write $X = (x_{1},x_{2},x_{3},x_{4},x_{5}) \in V$ to ease the notation. 
 The group $\Sp_{4}(\R)$ acts as isometries on $V$ by conjugation. In fact, the identity component $O(V)^{+}$ of the orthogonal group of $V$ is isomorphic to $\Sp_{4}(\R)/\{\pm 1\}$.

Let $\H_{2}$ be the Siegel upper half-space of genus $2$. For $Z = X + iY = \left(\begin{smallmatrix}z_{1} & z_{2} \\ z_{2} & z_{3}\end{smallmatrix}\right) \in \H_{2}$ we let
\[
X(Z) = \frac{1}{\sqrt{\det(Y)}}\begin{pmatrix}z_{2} & -z_{1} & 0 &\det(Z) \\ z_{3} & -z_{2} & -\det(Z)& 0 \\ 0 & -1 & z_{2} & z_{3} \\ 1 & 0 & -z_{1} & -z_{2} \end{pmatrix} \in V(\C).
\]
Note that $X(Z)$ has norm $0$, and that the real and the imaginary part have norm $1$ and are orthogonal. The map
\[
Z \mapsto \text{span}\big(\real X(Z), \imag X(Z)\big)
\]
gives a bijection between $\H_{2}$ and the Grassmannian $\Gr(V)$ of positive definite planes in $V$, which is compatible with the corresponding actions of $\Sp_{4}(\R)$. Note that the Siegel upper half-plane $\H_{2}$ can be naturally identified with the orthogonal half-plane $\mathcal{H}_{3}$ corresponding to the primitive isotropic vector $z = (1,0,0,0,0)$ and the vector $z' = (0,1,0,0,0)$. Thus orthogonal modular forms on $\mathcal{H}_{3}$ can be viewed as Siegel modular forms of genus $2$. 

Let $L$ be an even lattice in $V$, and let $\gamma \in L'/L$ and $n \in \Z + Q(\beta)$ with $m < 0$. In the Siegel upper half-space model of $\Gr(V)$, the Heegner divisor $H_{L}(\gamma,n)$ corresponds to the set
\[
\sum_{\substack{X \in \gamma +L \\ Q(X) = n}}\left\{\begin{pmatrix}z_{1} & z_{2} \\ z_{2} & z_{3} \end{pmatrix} \in \H_{2}: x_{2}(z_{2}^{2}-z_{1}z_{3})+x_{4}z_{1}-2x_{5}z_{2}+x_{3}z_{3}+x_{1} = 0\right\}.
\]

The ten even theta constants
\[
\vartheta_{a,b}(Z) = \sum_{g \in \Z^{2}}\exp\left(\pi i \left(Z[g+a/2]+b^{t}(g+a/2) \right)\right), 
\]
with $a,b \in \{0,1\}^{2}, \ a_{1}b_{1} + a_{2}b_{2} \equiv 0 \pmod 2$, are Siegel modular forms of weight $\frac{1}{2}$ for the principal congruence subgroup $\Gamma(2)$, see \cite{freitag}, Satz 3.2. The divisor of $\vartheta_{1,1,1,1}(Z)$ on $\H_{2}$ is given by $\Gamma_{\vartheta}\{Z \in \H_{2}: z_{2} = 0\}$, where $\Gamma_{\vartheta}$ is the theta group, see \cite{freitag}, Bemerkung~A~2.3. Since $\Sp_{4}(\Z)$ acts transitively on the even theta constants, we can easily determine the divisors of the other theta functions from this. The following result is well known.

\begin{Lemma} The divisor of $\vartheta_{a,b}(Z)$ on $\H_{2}$ is given by the set of all $Z \in \H_{2}$ satisfying an equation
\[
x_{2}(z_{2}^{2}-z_{1}z_{3}) + x_{4}z_{1}-2x_{5}z_{2} + x_{3}z_{3} = 0
\]
for some $(x_{1},\dots,x_{5}) \in \Z^{5}$ satisfying $x_{1}x_{2} + x_{3}x_{4} - x_{5}^{2} = -1$ and the following congruences mod $4$:
\begin{align*}
\begin{array}{c||c|c|c|c|c}
\vartheta_{a_{1},a_{2},b_{1},b_{2}} & x_{1}& x_{2} & x_{3} & x_{4} & x_{5} \\
\hline
\vartheta_{0,0,0,0} & 2 & 2 & 2 & 2 & \pm 1 \\
\vartheta_{0,0,0,1} & 0 & 2 & 2 & 0 & \pm 1 \\
\vartheta_{0,0,1,0} & 0 & 2 & 0 & 2 & \pm 1 \\
\vartheta_{0,0,1,1} & 0 & 2 & 0 & 0 & \pm 1 \\
\vartheta_{0,1,0,0}& 2 & 0 & 0 & 2 & \pm 1 \\
\vartheta_{0,1,1,0}& 0 & 0 & 0 & 2 & \pm 1 \\
\vartheta_{1,0,0,0}& 2 & 0 & 2 & 0 & \pm 1 \\
\vartheta_{1,0,0,1}& 0 & 0 & 2 & 0 & \pm 1 \\
\vartheta_{1,1,0,0}& 2 & 0 & 0 & 0 & \pm 1 \\
\vartheta_{1,1,1,1}& 0 & 0 & 0 & 0 & \pm 1 
\end{array}
\end{align*}
\end{Lemma}
We remark that the divisors of the ten even theta constants can be written as Heegner divisors with respect to the lattice $\sqrt{2}\Z^{5} \subset V$, but we chose the above formulation to make everything as explicit as possible. 

We now describe the Borcherds products of singular weight $\frac{1}{2}$ found in Theorem~\ref{theorem classification} in terms of theta constants. In each case, we first realize the simple lattice under consideration as a sublattice of $V$, which amounts to choosing a cusp at which we expand the Borcherds products for this lattice. We will frequently use the fact that, by the Koecher principle, a holomorphic Siegel modular form of weight $0$ for some finite index subgroup of $\Sp_{4}(\Z)$ and some multiplier system of finite order is constant. Hence, in order to show that our Borcherds products of weight $\frac{1}{2}$ are given by theta constants, it suffices to compare their divisors.

\subsection{The lattice $A_{1}(-4) \oplus U \oplus U$} We realize $L$ as the subset of $V$ consisting of those $X = (x_{1},\dots,x_{5}) \in V$ with $x_{1},\dots,x_{4} \in \Z$ and $x_{5} \in 2\Z$. Then the dual lattice $L'$ is then given by those $X \in V$ with $x_{1},\dots,x_{4} \in \Z$ and $x_{5} \in \frac{1}{4}\Z$. There are two good elements in $L'/L$, which are inverses of each other, namely $\pm \gamma = \pm(0,0,0,0,\frac{1}{4}) + L$. The corresponding Heegner divisor $H_{L}(\gamma,-\frac{1}{16})$ translates into the set
		\begin{align*}
		&\{Z \in \H_{2}: x_{2}(z_{2}^{2}-z_{1}z_{3})+x_{4}z_{1}-2x_{5}z_{2}+x_{3}z_{3}+x_{1} = 0 \, , \, x_{i} \in \Z, \, \\
		& \quad x_{1}x_{2}+x_{3}x_{4}-x_{5}^{2} = -1,\ x_{1} \equiv \ldots \equiv x_{4} \equiv 0 (4), \, x_{5} \equiv \pm 1 (8)  \}.
		\end{align*}
		This is the divisor of the theta constant $\vartheta_{1,1,1,1}(Z)$, which implies that the Borcherds product of weight $\frac{1}{2}$ with Heegner divisor $H_{L}(\gamma,-\frac{1}{16})$ equals $\vartheta_{1,1,1,1}(Z)$ up to multiplication by a constant.

\subsection{The lattice $A_{1}(-1) \oplus U(4) \oplus U$} We realize $L$ as the subset of $V$ with $x_{1},x_{3},x_{4},x_{5} \in \Z$ and $x_{2} \in 4\Z$. There is one good element of order $2$ in $L'/L$, namely $\gamma = (\frac{1}{2},2,0,0,\frac{1}{2}) + L$. By comparing the Heegner divisor $H_{L}(\gamma,-\frac{1}{4})$ in $\H_{2}$ to the divisors of the theta constants as above, we see that the corresponding Borcherds product is given by $\vartheta_{0,0,0,0}(2Z)$.

\subsection{The lattice $A_{1}(-1)\oplus U(4) \oplus U(2)$} We realize $L$ as the subset of $V$ with $x_{1},x_{3},x_{5} \in \Z$ and $x_{2} \in 4\Z, x_{4} \in 2\Z$. There are eight good elements $\gamma$ with order $2$ in $L'/L$. The corresponding Heegner divisors $H(\gamma,-\frac{1}{4})$ in $\H_{2}$ can be worked out and compared to the divisors of the theta constants as before. The resulting Borcherds products are given by
\begin{align*}
&\vartheta_{0,0,0,0}(2Z), \quad \vartheta_{0,0,0,0}\begin{pmatrix}4z_{1} & 2z_{2} \\ 2z_{2} & z_{3}\end{pmatrix}, \quad \vartheta_{0,0,1,0}(2Z), \quad  \vartheta_{0,0,0,1}\begin{pmatrix}4z_{1} & 2z_{2} \\ 2z_{2} & z_{3}\end{pmatrix}, \\ &\vartheta_{0,1,0,0}(2Z), \quad\vartheta_{1,0,0,0}\begin{pmatrix}4z_{1} & 2z_{2} \\ 2z_{2} & z_{3}\end{pmatrix}, \quad\vartheta_{0,1,1,0}(2Z), \quad \vartheta_{1,0,0,1}\begin{pmatrix}4z_{1} & 2z_{2} \\ 2z_{2} & z_{3}\end{pmatrix}.
\end{align*}

\subsection{The lattice $A_{1}(-2) \oplus U(2) \oplus U(2)$} We realize $L$ as the subset of $V$ with $x_{1},\dots,x_{5} \in \sqrt{2}\Z$. There are $20$ good elements in $L'/L$ which come in pairs $\pm \gamma$. The corresponding Borcherds products are exactly the ten even theta constants. This case has been treated in detail in the Diploma thesis of Lippolt \cite{lippolt}, written under the supervision of Freitag.

\subsection{The lattice $A_{1}(-1) \oplus U(4) \oplus U(4)$} We realize $L$ as the subset of $V$ with $x_{1},x_{2},x_{3},x_{4} \in 2\Z$ and $x_{5} \in \Z$. There are $10$ good elements of order $2$ in $L'/L$, which lead to the ten even theta constants, and $120$ good elements which do not have order $2$, and which form a single orbit under the action of $O(L'/L)$. For example, one pair of good elements is given by $\pm \gamma =\pm(1,0,1,0,\frac{1}{2})+L$, and the Borcherds product corresponding to the Heegner divisor $H_{L}(\gamma,-\frac{1}{4})$ is given by $\vartheta_{0,0,0,0}\left(\begin{smallmatrix}2z_{1} & z_{2} \\ z_{2} & \frac{z_{3}}{2} \end{smallmatrix}\right)$. The remaining Borcherds products can be determined analogously.

\section*{Appendix: Simple lattices}

We list the simple even lattices of signature $(2,n)$, $n \geq 3$. They have been determined by Bruinier, Ehlen and Freitag and can be found in the appendix of the extended online version \cite{bruinierehlenfreitagextended} of their journal article \cite{bruinierehlenfreitag}. 

Every genus in the following list contains exactly one isomorphy class. We describe the corresponding lattices in terms of the hyperbolic plane $U = (\Z^{2},Q(x,y) = xy)$, the standard positive definite root lattices $A_{n},D_{n}, E_{6}, E_{7}, E_{8}$ and the lattice 
\[
S_{8} = \begin{pmatrix}-8 & -4 & 0 \\ -4 & -2 & -1 \\ 0 & -1 & -2 \end{pmatrix}
\] 
with genus symbol $8_{3}^{-1}$. For a lattice $(L,Q)$ and an integer $N$ we let $L(N) = (L,NQ)$ denote the rescaled lattice. 

The simple even lattices of signature $(2,n)$, $n \geq 3$, are given in the following tables.
\begingroup
\renewcommand{\arraystretch}{1.2}
\begin{align*}
	\begin{array}{|c|l|l|}
	\hline
	n & \text{genus} & \text{lattice} \\
	\hline
	\hline
	3	& 2_{7}^{+1} 		& A_{1}(-1) \oplus U \oplus U \\
		& 2_{7}^{+3}			& A_{1}(-1) \oplus U(2) \oplus U \\
		& 2_{7}^{+1}4^{+2}	& A_{1}(-1) \oplus U(4) \oplus U \\
		& 2_{7}^{+5}			& A_{1}(-1) \oplus U(2) \oplus U(2) \\
		& 2_{7}^{+3}4^{+2}	& A_{1}(-1) \oplus U(2) \oplus U(4) \\
		& 2_{7}^{+1}4^{+4}	& A_{1}(-1) \oplus U(4) \oplus U(4) \\
		& 4_{7}^{+1}			& A_{1}(-2) \oplus U \oplus U \\
		& 2^{+2}4_{7}^{+1}	& A_{1}(-2) \oplus U(2) \oplus U \\
		& 2^{+4}4_{7}^{+1}	& A_{1}(-1) \oplus U(2) \oplus U(2) \\
		& 2_{1}^{+1}3^{+1}	& A_{1}(-3) \oplus U \oplus U \\
		& 2_{7}^{+1}3^{-2}	& A_{1}(-1) \oplus U(3) \oplus U \\
		& 2_{7}^{+1}3^{+4}	& A_{1}(-1) \oplus U(3) \oplus U(3) \\
		& 8_{7}^{+1}			& A_{1}(-4) \oplus U \oplus U \\
		& 8_{3}^{-1}			& S_{8} \oplus U  \\
		& 2^{+2}8_{3}^{-1}	& S_{8} \oplus U(2) \\
	\hline
	4	& 3^{+1}				& A_{2}(-1) \oplus U \oplus U \\
		& 3^{-3}				& A_{2}(-1) \oplus U(3) \oplus U \\
		& 3^{+5}				& A_{2}(-1) \oplus U(3) \oplus U(3) \\
		& 2^{+2}3^{+1}		& A_{2}(-1) \oplus U(2) \oplus U \\
		& 2^{+4}3^{+1}		& A_{2}(-1) \oplus U(2) \oplus U(2) \\
	\hline	
	\end{array}
	\qquad
	\begin{array}{|c|l|l|}
	\hline
	n & \text{genus} & \text{lattice} \\
	\hline
	\hline
	5	& 4_{5}^{-1}			& A_{3}(-1) \oplus U \oplus U \\
		& 2^{+2}4_{5}^{-1}	& A_{3}(-1) \oplus U(2) \oplus U \\
		& 2^{+4}4_{5}^{-1}	& A_{3}(-1) \oplus U(2) \oplus U(2) \\
	\hline 
	6	& 2^{-2}				& D_{4}(-1) \oplus U \oplus U \\
		& 2^{-4}				& D_{4}(-1) \oplus U(2) \oplus U \\
		& 2^{-6}				& D_{4}(-1) \oplus U(2) \oplus U(2) \\
		& 5^{+1}				& A_{4}(-1) \oplus U \oplus U \\
	\hline
	7	& 4_{3}^{-1}			& D_{5}(-1) \oplus U \oplus U \\
		& 2_{1}^{+1}3^{-1}	& A_{5}(-1) \oplus U \oplus U \\
	\hline
	8 	& 3^{-1}				& E_{6}(-1) \oplus U \oplus U \\
		& 2_{2}^{+2}			& D_{6}(-1) \oplus U \oplus U \\
		& 7^{+1}				& A_{6}(-1) \oplus U \oplus U \\
	\hline
	9	& 2_{1}^{+1}			& E_{7}(-1) \oplus U \oplus U \\
		& 4_{1}^{+1}			& D_{7}(-1) \oplus U \oplus U \\
		& 8_{1}^{+1}			& A_{7}(-1) \oplus U \oplus U \\
	\hline
	10	& 1^{+1}				& E_{8}(-1) \oplus U \oplus U \\
		& 2^{+2}				& E_{8}(-1) \oplus U(2) \oplus U \\
	\hline
	18	& 1^{+1}				& 2E_{8}(-1) \oplus U \oplus U \\
	\hline
	26	& 1^{+1}				& 3E_{8}(-1)\oplus U \oplus U \\
	\hline
	\end{array}
\end{align*}
\endgroup

\bibliography{references}

\begin{thebibliography}{BEF16b}

\bibitem[BEF16a]{bruinierehlenfreitag}
Jan~H. Bruinier, Stephan Ehlen, and Eberhard Freitag.
\newblock Lattices with many {B}orcherds products.
\newblock {\em Math. of. Comp.}, 85:1953--1981, 2016.

\bibitem[BEF16b]{bruinierehlenfreitagextended}
Jan~H. Bruinier, Stephan Ehlen, and Eberhard Freitag.
\newblock Lattices with many {B}orcherds products, extended version.
\newblock {\em Retrieved March 13, 2018, from
  https://github.com/sehlen/sfqm/blob/master/bruinier\_ehlen\_freitag\_extended.pdf},
  2016.

\bibitem[BK01]{BruinierKuss2001}
Jan~Hendrik Bruinier and Michael Kuss.
\newblock Eisenstein series attached to lattices and modular forms on
  orthogonal groups.
\newblock {\em Manuscripta Math.}, 106(4):443--459, 2001.

\bibitem[BM17]{BruinierMoeller2017}
Jan~Hendrik Bruinier and Martin {M{\"o}ller}.
\newblock {Cones of Heegner divisors}.
\newblock {\em ArXiv e-prints}, May 2017.

\bibitem[Bor98]{borcherds}
Richard~E. Borcherds.
\newblock Automorphic forms with singularities on {G}rassmannians.
\newblock {\em Invent. Math.}, 132(3):491--562, 1998.

\bibitem[Bru02]{bruinierhabil}
Jan~H. Bruinier.
\newblock {\em Borcherds products on $O(2, l)$ and {C}hern classes of {H}eegner
  divisors}, volume 1780 of {\em Lecture Notes in Mathematics}.
\newblock Springer Berlin Heidelberg New York, 2002.

\bibitem[Bun01]{bundschuh}
Michael Bundschuh.
\newblock {\em \"Uber die {E}ndlichkeit der {K}lassenzahl gerader {G}itter der
  {S}ignatur $(2,n)$ mit einfachem {K}ontrollraum}.
\newblock Heidelberg PhD Thesis, 2001.

\bibitem[DHS15]{dittmannhagemeierschwagenscheidt}
Moritz Dittmann, Heike Hagemeier, and Markus Schwagenscheidt.
\newblock {A}utomorphic products of singular weight for simple lattices.
\newblock {\em Math. Zeitschrift}, 279:585--603, 2015.

\bibitem[Dit18]{dittmann}
Moritz Dittmann.
\newblock Reflective automorphic products of squarefree level.
\newblock {\em To appear in Trans. Amer. Math. Soc.}, 2018.

\bibitem[Fre83]{freitag}
Eberhard Freitag.
\newblock {\em Siegelsche Modulfunktionen}.
\newblock Grundlehren der mathematischen Wissenschaften 254. Springer Berlin
  Heidelberg New York, 1983.

\bibitem[KY10]{KudlaYang2010}
Stephen~S. Kudla and TongHai Yang.
\newblock Eisenstein series for {SL}(2).
\newblock {\em Sci. China Math.}, 53(9):2275--2316, 2010.

\bibitem[Lip08]{lippolt}
Denis Lippolt.
\newblock {Thetanullwerte 2. Grades als Borcherds-Produkte}.
\newblock Diplomarbeit, Universit\"at Heidelberg, 2008.

\bibitem[Opi18]{opitzthesis}
Sebastian Opitz.
\newblock {\em {Computation of Eisenstein series associated with discriminant
  forms}}.
\newblock PhD thesis, {Technische Universit\"at Darmstadt}, {\textit{in
  preparation}}, 2018.

\bibitem[Sch06]{scheithauerclassification}
Nils Scheithauer.
\newblock On the classification of automorphic products and generalized
  {K}ac-{M}oody algebras.
\newblock {\em Invent. Math.}, 164:641--678, 2006.

\bibitem[Sch17]{scheithauerprimelevel}
Nils Scheithauer.
\newblock Automorphic products of singular weight.
\newblock {\em Compositio Math.}, 153:1855--1892, 2017.

\end{thebibliography}
\bibliographystyle{alpha}

\end{document}